 \documentclass[preprint,review,10pt]{elsarticle}

\usepackage{epsfig,amsmath,amssymb,amsthm,color,amsfonts,euscript,epstopdf}

\everymath{\displaystyle}
\def\l{\left}
\def\r{\right}

\def\LL{{\mathcal{L}}}
\def\RR{{\mathbb{R}}}
\def\DD{{\cal{D}}}
\def\OO{{\cal{O}}}
\def\SS{{\cal{S}}}
\def\R#1{$(\ref{#1})$}

\newcommand{\bb}[1]{\begin{equation}\label{#1}}
\newcommand{\ee}{\end{equation}}
\newcommand{\bbb}{\begin{eqnarray}}
\newcommand{\eee}{\end{eqnarray}}
\newcommand{\bbbb}{\begin{eqnarray*}}
\newcommand{\eeee}{\end{eqnarray*}}
\newcommand{\tT}{\intercal}
\newcommand{\nnn}{\nonumber}
\newcommand{\no}{\noindent}
\definecolor{green1}{rgb}{0.1,0.5,0.0}

\newtheorem{thm}{Theorem}
\newtheorem{lemma}{Lemma}
\newtheorem{cor}{Corollary}
\theoremstyle{remark}
\newtheorem{rem}{Remark}
\theoremstyle{define}
\newtheorem{define}{Definition}

\newcommand{\clearallnum}{
    \numberwithin{equation}{section} \setcounter{equation}{0}
    \numberwithin{thm}{section} \setcounter{thm}{0}
    \numberwithin{lemma}{section} \setcounter{lemma}{0}
    \numberwithin{cor}{section} \setcounter{cor}{0}
    \numberwithin{rem}{section} \setcounter{rem}{0}
    \numberwithin{define}{section} \setcounter{define}{0}}

\begin{document}


\begin{frontmatter}

\title{On the positivity, monotonicity, and stability of a semi-adaptive LOD method 
for solving three-dimensional degenerate Kawarada equations}
\author{Joshua L. Padgett\footnote{Principal and corresponding author.
Email address: Josh\underline{~}Padgett@baylor.edu} and 
Qin Sheng\footnote{The second author is supported in part by a URC Research Award 
(No. 3033-0248/2014) from Baylor University.}}
\address{Department of Mathematics and
Center for Astrophysics, Space Physics and Engineering Research\\
Baylor University, One Bear Place, Waco, TX 76798-7328}

\begin{abstract}
This paper concerns the numerical solution of three-dimensional degenerate Kawarada
equations. These partial differential equations possess highly nonlinear source terms, and 
exhibit strong quenching singularities which pose severe challenges to the design and 
analysis of highly reliable schemes.  Arbitrary fixed nonuniform spatial grids, which are 
not necessarily symmetric, are considered throughout this study. The numerical solution 
is advanced through a semi-adaptive Local One-Dimensional (LOD) integrator. The 
temporal adaptation is achieved via a suitable arc-length monitoring mechanism. Criteria 
for preserving the positivity and monotonicity are investigated and acquired. The 
numerical stability of the splitting method is proven in the von Neumann sense 
under the spectral norm. Extended stability expectations are proposed and investigated.
\end{abstract}

\end{frontmatter}

\noindent AMS Subject Claasifications - 65K20, 65M50, 35K65, 35B40

\section{Introduction} \clearallnum
Let $\DD=(0,a)\times (0,b)\times (0,c)\subset\RR^3,$ where $a,b,c>0,$ and $\partial\DD$ be its boundary. 
Denote $\Omega=\DD\times (t_0,~T),~\SS=\partial \DD\times (t_0,~T)$ for given $0\le t_0<T<\infty.$ 
We consider the following degenerate Kawarada problem,
\bbb
&& s(x,y,z) u_t=u_{xx}+u_{yy}+u_{zz}+f(u),~~~(x,y,z,t)\in\Omega,\label{b1}\\
&& u(x,y,z,t)=0,~~~(x,y,z,t)\in \SS,\label{b2}\\
&& u(x,y,z,t_0)=u_0(x,y,z),~~~(x,y,z)\in \DD,\label{b3}
\eee
where $s(x,y,z)=\l(x^2+y^2+z^2\r)^{q/2},~q\in[0,2].$ 
The nonlinear source function, $f(u),$ 
is strictly increasing for $0\leq u<1$ with
$$f(0)=f_0>0,~~\lim_{u\rightarrow 1^-}f(u)=\infty.$$
In idealized thermal combustion applications \cite{Acker2,Bebernes_89,Sheng1}, 
$u$ represents the temperature in the combustion 
channel, and the $x$-, $y$-, and $z$-coordinates coincide with the channel walls. The initial temperature $0 \le u_0 \ll 1$ is 
typically chosen to be small. The function $s(x,y,z)$ represents certain singularities in the temperature transportation speed 
within the channel, which causes the degeneracy in the differential equation 
\R{b1} \cite{Beau1,Ock,Sheng21,Sheng3}. The solution 
$u$ of \R{b1}-\R{b3} is said to {\em quench\/} if there exists a finite time $T>0$ such that
\bb{a1}
\sup\l\{|u_t(x,y,z,t)|:(x,y,z)\in {\DD}\r\}\rightarrow\infty~\mbox{as}~t\rightarrow
T^{-}.
\ee
The value $T$ is then defined as the {\em quenching time\/} \cite{Acker1,Acker2,Levine}. 
It has been shown that a necessary condition for quenching to occur is
\bb{a2}
\max\l\{|u(x,y,z,t)|:(x,y,z)\in \bar{\DD}\r\}\rightarrow 1^{-}~\mbox{as}~t\rightarrow T^{-}.
\ee
Further, such a $T$ exists only when certain spatial references, such as the size and shape of $\DD,$ reach their critical 
limits. A domain $\DD^*$ is called the {\em critical domain\/} if the solution of \R{b1}-\R{b3} exists for all time when $\DD\subseteq 
\DD^*,$ and \R{a2} occurs when $\DD^*\subseteq \DD$ for a finite $T$ \cite{Levine}.

Systematic mathematical investigations of quenching phenomena can be traced back to Karawada's original work involving the
one-dimensional model equation \cite{Kawa}. It was observed that for any spatial domain $[0,a],$ there exists a unique value $a^*>0$ 
such that for $a<a^*,$ the solution of the equation exists globally; and for $a\geq a^*,$ there exists a finite time $T(a),$ 
such that $\lim_{t\rightarrow T(a)}\max_{0\leq x\leq a} u(x,t)=1.$ In the latter case, $u$ stops existing in finite time and this phenomenon 
is referred to as {\em quenching\/} \cite{Kawa,Levine,Sheng4}. 
There have been considerable developments in the study of Karawada equations, although discussions of 
multidimensional problems were extremely limited until recently.
In 1994, Chan and Ke proved that for a domain 
$\DD=(0,a)\times (0,b)\subset\RR^2,$ if $f,~f_u$ are nonnegative, then, for any fixed ratio $a/b,$ there exists a unique critical domain 
$\DD^*$ for \R{b1}-\R{b3}, and the solution of the differential equation problem is unique before quenching \cite{Chan2}. 
A numerical approximation of the relationship between $a/b$ and the areas of $\DD^*$ of a nondegenerate ($q=0$) problem 
was given. These results have been well supported by realistic physical processes, in particular in solid fuel combustion
\cite{Bebernes_89,Beau1,Sheng4}.

Numerous computational procedures, including moving mesh adaptive methods, have been constructed for solving 
blow-up and Kawarada problems in the past decades (interested readers are referred to 
\cite{Acker1,Cheng,Coyle,Sheng21,Sheng3} and 
references therein). Though in the former case, adaptations are frequently achieved via monitoring functions on the arc-length 
of the function $u;$ in the latter situation, adaptations are more likely to be built upon the arc-length of $u_t,$ since
it is directly proportional to $f(u),$ which blows up as $u$ quenches \cite{Chan2,Levine,Sheng15b}.

As reported in several recent investigations, when quenching locations can be predetermined, it is preferable 
to use nonuniform spatial grids throughout the computations \cite{Beau1,Lang2,Sheng3}. 
In this case,
key quenching characteristics such as the quenching time and critical domain, 
are more easily observed; Also important numerical properties of underlying algorithms, 
including the monotonicity, stability and
convergence, can be more precisely studied. 
To that end, this paper develops a temporally adaptive splitting 
scheme utilizing predetermined nonuniform spatial grids. The positivity, monotonicity, and stability of the method will be investigated. It is also observed that the impact of degeneracy is limited for
our implicit scheme. Our discussions will be organized as follows. In the next section, 
the semi-adaptive LOD scheme for solving \R{b1}-\R{b3} will be constructed and discussed. Then, in Section 3, 
criteria to guarantee the positivity of the numerical scheme will be determined. In Section 4, 
appropriate criteria for guaranteeing the monotonicity will be obtained. These two sections together serve
as the platform for carrying out investigations of stability. Section 5 is devoted to the 
stability analysis of the semi-adaptive LOD scheme. The analysis will first be carried out for a 
fully linearized scheme, and then a more realistic stability analysis is proposed without freezing the source term. 
Finally, concluding remarks and 
proposed future work will be given in Section 6. For now, no numerical studies of the three-dimensional degenerate Karawada problem will be given.

\section{Semi-adaptive LOD scheme} \clearallnum
Utilizing the transformations $\tilde{x}=x/a,~\tilde{y}=y/b,~\tilde{z}=z/c,$ and reusing the original variables 
for simplicity, we may reformulate \R{b1}-\R{b3} as
\bbb
&&u_t=\frac{1}{a^2\phi}u_{xx} + \frac{1}{b^2\phi}u_{yy} +
\frac{1}{c^2\phi}u_{zz} +g(u),~~(x,y,z,t)\in\Omega,~~~~~~\label{c1}\\
&&u(x,y,z,t)=0,~~(x,y,z)\in \SS,\label{c2}\\
&&u(x,y,z,t_0)=u_0,~~(x,y,z)\in \DD,\label{c2b}
\eee
where $g(u) = f(u)/\phi,~\phi=\phi(x,y,z) = \l(a^2x^2+b^2y^2+c^2z^2\r)^{q/2},$ and 
$\DD=(0,1)\times (0,1)\times (0,1)\subset\RR^3.$

Let $N_1,N_2,N_3\gg 1.$ We inscribe over $\bar{\DD}$ the following variable grid: 
$\DD_h=\l\{(x_i,y_j,z_k)| i=0,\dots,N_1+1; ~j=0,\dots,N_2+1;~k=0,\dots,N_3+1;
~x_0=y_0=\r.$\\$\l.z_0=0,~x_{N_1+1}=y_{N_2+1}\r.$ $\l.=z_{N_3+1}=1\r\}.$ 
Denote $h_{1,i} = x_{i+1}-x_i>0,~h_{2,j} = y_{j+1}-y_j>0,$ and 
$h_{3,k} = z_{k+1}-z_k>0$ for $1\leq i\leq N_1,~1\leq j\leq N_2,~1\leq k\leq N_3.$ 
Let $u_{i,j,k}(t)$ be an approximation of the solution of \R{c1}-\R{c2b} at  
$(x_i,y_j,z_k,t)$ and consider the following first-order finite differences \cite{Sheng3},
\bbbb
\l.\frac{\partial^2u}{\partial x^2}\r|_{i,j,k} &\approx& \frac{2u_{i-1,j,k}}{h_{1,i-1}(h_{1,i-1}+h_{1,i})}-\frac{2u_{i,j,k}}{h_{1,i-1}h_{1,i}}+\frac{2u_{i+1,j,k}}{h_{1,i}(h_{1,i-1}+h_{1,i})},\\
\l.\frac{\partial^2u}{\partial y^2}\r|_{i,j,k} &\approx& \frac{2u_{i,j-1,k}}{h_{2,j-1}(h_{2,j-1}+h_{2,j})}-\frac{2u_{i,j,k}}{h_{2,j-1}h_{2,j}}+\frac{2u_{i,j+1,k}}{h_{2,j}(h_{2,j-1}+h_{2,j})},\\
\l.\frac{\partial^2u}{\partial z^2}\r|_{i,j,k} &\approx& \frac{2u_{i,j,k-1}}{h_{3,k-1}(h_{3,k-1}+h_{3,k})}-\frac{2u_{i,j,k}}{h_{3,k-1}h_{3,k}}+\frac{2u_{i,j,k+1}}{h_{3,k}(h_{3,k-1}+h_{3,k})}.
\eeee
Further, denote
$v(t)=(u_{1,1,1},u_{2,1,1},\dots,u_{N_1,1,1},
u_{1,2,1},u_{2,2,1},\dots,u_{N_1,2,1},\ldots,$\\$
u_{1,N_2,1},u_{2,N_2,1},\dots,u_{N_1,N_2,1},
\ldots, u_{1,N_2,N_3},u_{2,N_2,N_3},\dots, u_{N_1,N_2,N_3})^{\tT}\in\RR^{N_1N_2N_3}$
and let $g(v)$ be a discretization of the nonhomogeneous term of \R{c1}. 
We obtain readily from \R{c1}-\R{c2b} the following semi-discretized system
\bbb
v'(t)&=& \sum_{\sigma=1}^3 M_{\sigma} v(t)+g(v(t)),~~~t_0<t<T,\label{c8}\\
v(t_0)&=&v_0,\label{c8b}
\eee
where
$$M_1=\frac{1}{a^2}B(I_{N_3}\otimes I_{N_2}\otimes T_1),~M_2=\frac{1}{b^2}B(I_{N_3}\otimes T_2
\otimes I_{N_1}),~M_3=\frac{1}{c^2}B(T_3\otimes I_{N_2}\otimes I_{N_1}),$$
$\otimes$ stands for the Kronecker product,
$I_{N_{\sigma}}\in\RR^{N_{\sigma}\times N_{\sigma}},~\sigma=1,2,3,$ are identity matrices, and
\bbbb
B&=&\mbox{diag}\l(\phi_{1,1,1}^{-1},\phi_{2,1,1}^{-1},...,\phi_{N_1,1,1}^{-1},\phi_{1,2,1}^{-1},...,
\phi_{N_1,N_2,N_3}^{-1}\r)\in\RR^{N_1N_2N_3\times N_1N_2N_3},\\
\phi_{i,j,k}&=& \l[a^2\l(\sum_{\ell=0}^{i-1}h_{1,\ell}\r)^2+b^2\l(\sum_{\ell=0}^{j-1}h_{2,\ell}\r)^2+
c^2\l(\sum_{\ell=0}^{k-1}h_{3,\ell}\r)^2\r]^{q/2},\\
T_{\sigma}&=&\l(\begin{array}{rrrrr}m_{\sigma,1}& n_{\sigma,1} &~&~&~\\ l_{\sigma,1} & m_{\sigma,2}& 
n_{\sigma,2}&~&~\\~&\cdots&\cdots&\cdots&~\\ ~&~& l_{\sigma,N_{\sigma}-2}& m_{\sigma,N_{\sigma}-1}& 
n_{\sigma,N_{\sigma}-1}\\ ~&~&~& l_{\sigma,N_{\sigma}-1}& m_{\sigma,N_{\sigma}}\end{array}\r)\in
\RR^{N_{\sigma}\times N_{\sigma}},~~~\sigma=1,2,3,
\eeee
and for the above
\bbbb
l_{\sigma,j} &=& \frac{2}{h_{\sigma,j}(h_{\sigma,j}+h_{\sigma,j+1})},~n_{\sigma,j} ~=~ \frac{2}{h_{\sigma,j}(h_{\sigma,j-1}
+h_{\sigma,j})},~~~j = 1,\dots,N_{\sigma}-1,\\
m_{\sigma,j} &=& -\frac{2}{h_{\sigma,j-1}h_{\sigma,j}},~~~j=1,\dots,N_{\sigma};~\sigma=1,2,3.
\eeee
The formal solution of \R{c8}, \R{c8b} can thus be written as
\bb{rs1}
v(t)=E(tC)v_0+\int_{t_0}^tE((t-\tau)C)g(v(\tau))d\tau,~~~t_0<t<T,
\ee
where $E(\cdot)=\exp(\cdot)$ is the matrix exponential and $C=\sum_{\sigma=1}^3M_{\sigma}$ \cite{Sheng1}.

In principle, different approximation techniques can be used to 
yield different splitting methods based on \R{rs1} \cite{Gold,Sheng1,Sheng15b}. 
Yet, we are particularly interested in approximating \R{rs1} via a trapezoidal rule and a
[1/1] Pad\'{e} approximation, $E(tC) = p(t) + \OO\l(t^2\r),$ 
where
$$
p(t)=\prod_{\sigma=1}^3\l(I-\frac{t}{2}M_{\sigma}\r)^{-1}\l(I+\frac{t}{2}M_{\sigma}\r),~~~t_0<t<T.
$$
The above leads to 
\bb{cc3}
v(t)=p(t)\l[v_0 + \frac{t}{2}g(v_0)\r] + \frac{t}{2}g(v(t)) + \OO\l((t-t_0)^2\r),~~~t\rightarrow t_0.
\ee
The above LOD algorithm provides a highly efficient way to compute numerical solutions of
multidimensional problems such as \R{c1}-\R{c2b} \cite{Gold,Scha,Sheng21,Sheng3}. 
Based on \R{cc3}, we obtain the 
following first order in space and time semi-adaptive LOD scheme: 
\bb{c3}
v_{\ell+1}=\l[\prod_{\sigma=1}^3\l(I-\frac{\tau_\ell}{2}M_{\sigma}\r)^{-1}
\l(I+\frac{\tau_\ell}{2}M_{\sigma}\r)\r]\l(v_\ell+\frac{\tau_\ell}{2}g(v_\ell)\r)+
\frac{\tau_\ell}{2}g(v_{\ell+1}),
\ee
where $v_{\ell}$ and $v_{\ell+1}$ are approximations of $v(t_{\ell})$ and $v(t_{\ell+1}),$ respectively, $v_0$
is the initial vector,
$t_{\ell}=t_0+\sum_{k=0}^{\ell-1}\tau_k,~\ell=0,1,2,\ldots,$ and $\{\tau_{\ell}\}_{\ell\ge 0}$ is a set of variable temporal steps 
determined by an adaptive procedure. In order to avoid a fully implicit scheme, $g(v_{\ell+1})$ may be
approximated by $g(w_{\ell}),$ where $w_{\ell}$ is an approximation to $v_{\ell+1},$ 
such as 
\bb{approx}
w_{\ell} = v_{\ell} + \tau_{\ell}(Cv_{\ell}+g(v_{\ell})),~~~0<\tau_{\ell}\ll 1,
\ee
in practical computations.

\begin{center}
{\epsfig{file=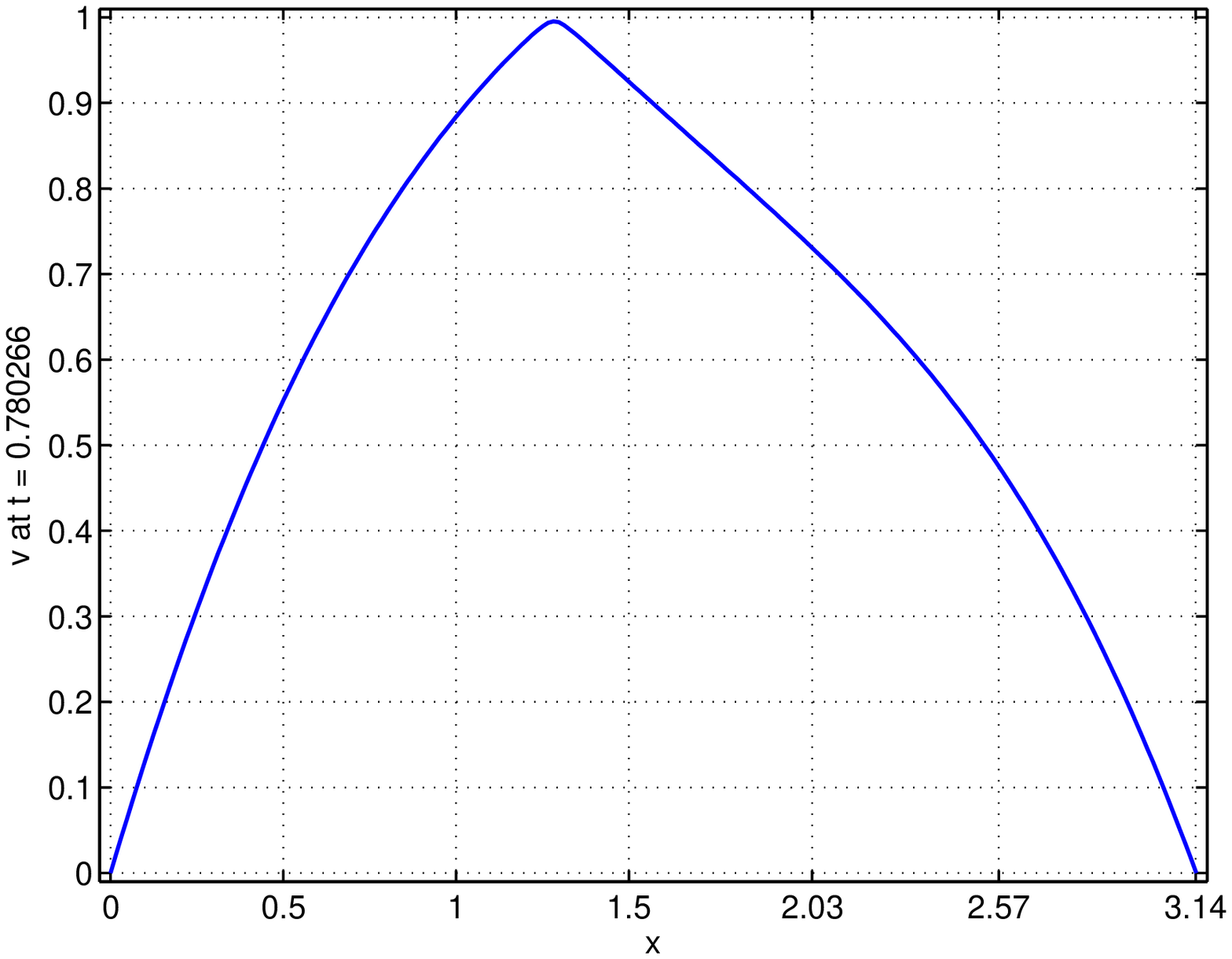,width=2.2in,height=1.38in}}
{\epsfig{file=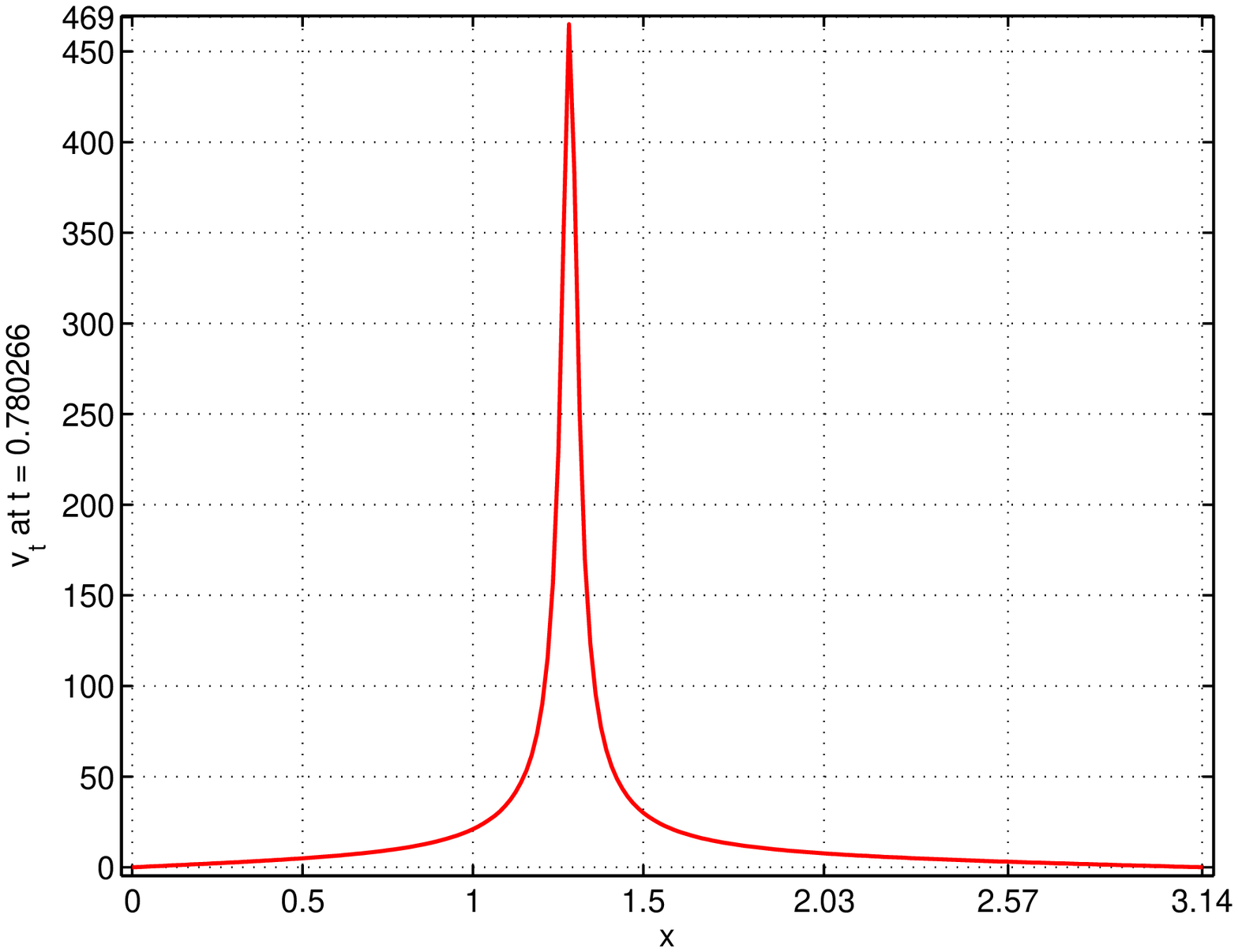,width=2.2in,height=1.38in}}

\parbox[t]{11.8cm}{\scriptsize{\bf Figure 1.} ~Numerical
solution (left) and its temporal derivative (right) immediately before quenching.
It is observed that as $\max_{x}v(x)\rightarrow 1^-,$ we have $\max_{x}v_t\gg 600.$
The computed quenching time is $T\approx0.780265747310047.$}
\end{center}

Due to a strong quenching singularity, the selection of proper nonuniform 
temporal steps $\tau_\ell$ is vital. As an illustration, in Figure 1, we show the numerical solution and
its temporal derivative of a typical one-dimensional Kawarada problem over the interval $[0, \pi].$ The initial function $u_0(x)=0.001\sin(x),~f(u)=1/(1-u),$ and homogeneous Dirichlet boundary condition
are employed. The degeneracy function utilized is $s(x)=x^p(\pi-x)^{1-p}$ with 
$p=(\sqrt{5}-1)/2.$ It is
evident that $v_t$ changes dramatically when compared with $v.$
Recalling \R{a1} and \R{a2}, we consider the following arc-length monitoring function on $v_t,$  
$$m\l(\frac{\partial v}{\partial t},t\r)=\sqrt{1+\l(\frac{\partial^2 v}{\partial t^2}\r)^2},~~~t_0<t<T.$$
Setting the two maximal arc-lengths in neighboring intervals $[t_{\ell-2},~t_{\ell-1}]$ and $[t_{\ell-1},~t_\ell]$ equal
\cite{Fur,Lang2,Sheng4,Sheng3}, we acquire the following quadratic equations from the above, 
$$\tau_\ell^2=\tau_{\ell-1}^2+\l(\frac{\partial v_{\ell-1}}{\partial t}-\frac{\partial v_{\ell-2}}{\partial t}\r)^2-
\l(\frac{\partial v_\ell}{\partial t}-\frac{\partial v_{\ell-1}}{\partial t}\r)^2,~~~\ell=1,2,3,\ldots,$$
with $\tau_0$ given. 

In the above temporal adaptation procedures, we may consider a minimal temporal step size controller 
$\tilde{\tau}_0,~0<\tilde{\tau}_0\ll\tau_0,$ to avoid sudden changes in grid 
movements or unnecessarily large numbers of computations. Further, let $\wedge$ be one 
of the operations $<,~\leq,~>,~\geq$ and $\alpha,\beta\in\RR^N.$ We assume the 
following notations in subsequent discussions:

\vspace{-2mm}

\begin{enumerate}
\item $\alpha\wedge\beta$ means $\alpha_i\wedge\beta_i,~i=1,2,\ldots,N;$
\vspace{-2mm}
\item $a\wedge\alpha$ means $a\wedge\alpha_i,~~i=1,2,\ldots,N,$
for any given scalar $a.$
\end{enumerate}

\section{Positivity} \clearallnum
The positivity property is one of the most profound characteristics of the solution of the Kawarada problem
\R{b1}-\R{b3} or \R{c1}-\R{c2b} \cite{Acker1,Acker2,Chan2,Levine}. Since positive computational solutions 
preserve the correct physical features of quenching phenomena, it is crucial that our numerical 
solution also possesses this feature.

\vspace{3mm}

\begin{lemma}
$\|T_{\sigma}\|_2 \le\max_{j=1,\dots,N_{\sigma}}\l\{{4}/{h_{\sigma,j}^2}\r\},~~~\sigma=1,2,3$
\end{lemma}

\begin{proof}
Due to the similarity in structure, we only consider $T_1$ since the other two cases follow by similar arguments. Note that, in general, $T_1$ is not symmetric. However, $\|T_1\|_2^2 = \rho(T_1^{\tT}T_1)$ and $T_1^{\tT}T_1$ is symmetric with a bandwidth of five. Thus,
$$T_1^{\tT}T_1 = \l(\begin{array}{ccccccc} \tilde{m}_{1,1} & \tilde{n}_{1,1} & \tilde{N}_{1,1} & & & & \\ \tilde{l}_{1,1} & \tilde{m}_{1,2} & \tilde{n}_{1,2} & \tilde{N}_{1,2} & & & \\ \tilde{L}_{1,1} & \tilde{l}_{1,2} & \tilde{m}_{1,3} & \tilde{n}_{1,3} & \tilde{N}_{1,3} & & \\  & & \cdots & \cdots & \cdots & \cdots & \\ & & \tilde{L}_{1,N_1-4} & \tilde{l}_{1,N_1-3} & \tilde{m}_{1,N_1-2} & \tilde{n}_{1,N_1-2} & \tilde{N}_{1,N_1-2}\\ & & & \tilde{L}_{1,N_1-3} & \tilde{l}_{1,N_1-2} & \tilde{m}_{1,N_1-1} & \tilde{n}_{1,N_1-1}\\ & & & & \tilde{L}_{1,N_1-2} & \tilde{l}_{1,N_1-1} & \tilde{m}_{1,N_1} \end{array}\r)\in \RR^{N_1\times N_1},$$
where
\bbbb
\tilde{N}_{1,j} &=& \tilde{L}_{1,j} ~=~ l_{1,j} n_{1,j+1},~~~j=1,\dots,N_1-2,\\
\tilde{n}_{1,j} &=& \tilde{l}_{1,j} ~=~ m_{1,j} n_{1,j} + m_{1,j+1}l_{1,j},~~~j=1,\dots,N_1-1,\\
\tilde{m}_{1,j} &=& \l\{\begin{array}{ll}
m_{1,1}^2 + l_{1,1}^2, & j=1,\\
n_{1,j-1}^2 + m_{1,j}^2 + l_{1,j}^2, & j=2,\dots,N_1-1,\\
n_{1,N_1-1}^2+m_{1,N_1}^2, & j=N_1.
\end{array}\r.
\eeee
We may determine a bound on the spectral radius of $T_1^{\tT}T_1$ by using Ger\u{s}chgorin's circle theorem. In fact, 
only rows containing five nontrivial elements, {\em i.e., $j=3,\dots,N_1-2,$\/} need to be considered. To this end,
$$|\lambda_{1,j} - \tilde{m}_{1,j}|\le |\tilde{L}_{1,j-2}|+|\tilde{l}_{1,j-1}|+|\tilde{n}_{1,j}|+|\tilde{N}_{1,j}|,~~~j=3,\dots,N_1-2,$$
which gives
\bbb
-|m_{1,j-1}n_{1,j-1}+m_{1,j}l_{1,j-1}|-|m_{1,j}n_{1,j}+m_{1,j+1}l_{1,j}| &&\nnn\\
- |l_{1,j-2}n_{1,j-1}| - |l_{1,j}n_{1,j+1}| + n_{1,j-1}^2 + m_{1,j}^2 + l_{1,j}^2 & \le & \lambda_{1,j}\label{g1}
\eee
and
\bbb
\lambda_{1,j} & \le & |m_{1,j-1}n_{1,j-1}+m_{1,j}l_{1,j-1}|+|m_{1,j}n_{1,j}+m_{1,j+1}l_{1,j}| \nnn\\
&& + |l_{1,j-2}n_{1,j-1}| + |l_{1,j}n_{1,j+1}| + n_{1,j-1}^2 + m_{1,j}^2 + l_{1,j}^2 . \label{g2}
\eee
Let $h_1\equiv \min_{j=1,\dots,N_1} \{h_{1,j}\}.$ From \R{g2} we acquire that
\bbbb
\lambda_{1,j} & \le & \frac{2}{h_1^2}\cdot\frac{2}{2h_1^2} + \frac{2}{h_1^2}\cdot\frac{2}{2h_1^2} + \frac{2}{h_1^2}\cdot\frac{2}{2h_1^2} + \frac{2}{h_1^2}\cdot\frac{2}{2h_1^4} + \frac{4}{4h_1^4} + \frac{4}{4h_1^2}\\
&& + \l(\frac{2}{2h_1^2}\r)^2 + \l(\frac{2}{h_1^2}\r)^2 + \l(\frac{2}{2h_1^2}\r)^2~ = ~ \frac{16}{h_1^4}.
\eeee
Now, reverse \R{g1} and by the same token,
\bbbb
-\lambda_{1,j} & \le & |m_{1,j-1}n_{1,j-1}+m_{1,j}l_{1,j-1}|+|m_{1,j}n_{1,j}+m_{1,j+1}l_{1,j}| + |l_{1,j-2}n_{1,j-1}|\\
&&+ |l_{1,i}n_{1,j+1}| - n_{1,j-1}^2 - m_{1,j}^2 - l_{1,j}^2\\
& \le & |m_{1,j-1}n_{1,j-1}+m_{1,j}l_{1,j-1}|+|m_{1,j}n_{1,j}+m_{1,j+1}l_{1,j}| + |l_{1,j-2}n_{1,j-1}|\\
&&+ |l_{1,i}n_{1,j+1}| + n_{1,j-1}^2 + m_{1,j}^2 + l_{1,j}^2~\le~ \frac{16}{h_1^4}\;,
\eeee
where, once again, $h_1\equiv \min_{j=1,\dots,N_1} \{h_{1,j}\}.$ Thus, combining the bounds we have 
$\|T_1\|_2 \le \max_{i=1,\dots,N_1}\l\{{4}/{h_{1,j}^2}\r\}.$ The other bounds follow similarly.
\end{proof}

\begin{lemma}
Let
\bbbb
\beta_{\min} &=& \frac{h^2}{2\|B\|_2},~h~ =~ \min_{j=1,\dots,N_{\sigma};\;\sigma=1,2,3}
\{h_{\sigma,j}\},\\
\frac{1}{\|B\|_2} &=& \min_{i,j,k} \phi_{i,j,k} = \l[a^2h_{1,0}^2+b^2h_{2,0}^2+c^2h_{3,0}^2\r]^{q/2}.
\eeee
If
\bb{cfl}
\frac{\tau_\ell}{\beta_{\min}}<\min\l\{a^2,b^2,c^2\r\},
\ee
then the matrices
$$I-\frac{\tau_\ell}{2}M_\sigma,~I+\frac{\tau_\ell}{2}M_\sigma,~~~\sigma=1,2,3,$$
are nonsingular. Further, $I-\frac{\tau_\ell}{2}M_\sigma,~\sigma=1,2,3,$ are monotone and 
inverse-positive, and $I+\frac{\tau_\ell}{2}M_\sigma,~\sigma=1,2,3,$ are nonnegative.
\end{lemma}

\begin{proof}
First, note that
\bbbb
\l\|\frac{\tau_\ell}{2}M_1\r\|_2 & = & \frac{\tau_\ell}{2a^2}\|B(I_{N_3}\otimes I_{N_2}\otimes T_1)\|_2\\
& \le & \frac{\tau_\ell}{2a^2}\|B\|_2\|I_{N_3}\otimes I_{N_2}\otimes T_1\|_2~=~ \frac{\tau_\ell}{2a^2}\|B\|_2\|T_1\|_2\\
& \le & \frac{\tau_\ell}{a^2}\|B\|_2\max_{j=1,\dots,N_1}\l\{\frac{2}{h_{1,j}^2}\r\}
~ < ~ 1.
\eeee
Hence, $I+\frac{\tau_\ell}{2}M_1$ is nonsingular, and also nonnegative. Similar arguments give that $I+\frac{\tau_\ell}{2}M_2$ and $I+\frac{\tau_\ell}{2}M_3$ are nonsingular and nonnegative.

Now, consider $A = I-\frac{\tau_\ell}{2}M_1.$ As $A_{ij}\le 0$ for $i\neq j$ and the weak row sum criterion is satisfied, $A$ is monotone, and hence an inverse exists and is nonnegative. So, $A$ must be
inverse-positive \cite{Golub}. Similar arguments can be given for
$I-\frac{\tau_\ell}{2}M_2$ and $I-\frac{\tau_\ell}{2}M_3.$ This ensures the proof.
\end{proof}

We also need the following lemma.

\begin{lemma}
Let $A\in\RR^{n\times n}$ be nonsingular and nonnegative 
and $\beta\in\RR^{n}$ be positive. Then $A \beta > 0.$
\end{lemma}

\begin{proof}
The proof is a straightforward application of the definitions.
\end{proof}

\section{Monotonicity} \clearallnum
Another key characteristic which distinguishes a solution to a quenching problem from a solution to most blow-up problems is its monotonicity with respect to time $t\geq t_0$ \cite{Acker1,Acker2,Chan2,Levine,Sheng21}.
Thus, it is necessary to guarantee that our numerical solution preserves this property strictly while
solving the Kawarada equation \R{b1}-\R{b3} or \R{c1}-\R{c2b}.


\begin{lemma}
If \R{cfl} holds for all $\ell\ge k\geq 0,$ and
\begin{itemize}
\item[\rm(a)] $Cv_0 + \frac{1}{2}g(v_0) >0;$
\item[\rm(b)] $\l(I-\frac{\tau_0}{2}g_v(\xi_0)\r)^{-1} >0$
\end{itemize}
hold, then $v_{\ell+1}\ge v_\ell \mbox{ for all } \ell\ge 0.$ That is, the sequence $\l\{v_\ell\r\}_{\ell=0}^{\infty}$ is monotonically increasing.
\end{lemma}

\begin{proof}
By \R{cfl} we have $\l\|\frac{\tau_k}{2}M_\sigma\r\|<1,$ and thus,
$$\l(I-\frac{\tau_k}{2}M_\sigma\r)^{-1} = I + \frac{\tau_k}{2}M_\sigma + \OO\l(\tau_k^2\r),~~~\sigma=1,2,3.$$
From \R{c3} and the above, we have
\bbb
v_{k+1} - v_k & = & \l[\prod_{\sigma=1}^3\l(I-\frac{\tau_k}{2}M_\sigma\r)^{-1}\l(I+\frac{\tau_k}{2}
M_\sigma\r)\r]\l(v_k+\frac{\tau_k}{2}g(v_k)\r)+\frac{\tau_k}{2}g(v_{k+1})-v_k\nnn\\
& = & \l[\prod_{\sigma=1}^3\l(I+\tau_k
M_\sigma\r)+\OO\l(\tau_k^2\r)\r]\l(v_k+\frac{\tau_k}{2}g(v_k)\r)+\frac{\tau_k}{2}g(v_{k+1})-v_k\nnn\\
& = & \l[\l(I+\tau_kC\r)+\OO\l(\tau_k^2\r)\r]\l(v_k+\frac{\tau_k}{2}g(v_k)\r)+\frac{\tau_k}{2}g(v_{k+1})-v_k\nnn\\
& = & \frac{\tau_k}{2}g(v_k) + \tau_kCv_k + \frac{\tau_k}{2}g(v_{k+1}) + \OO\l(\tau_k^2\r)\label{m1}
\eee
as $\tau_k\rightarrow 0.$ Note that $g(v_{k+1}) = g(v_k) + g_v(\xi_k)(v_{k+1}-v_k)$ for some $\xi_k\in \LL(v_{k+1};v_k),$ 
where $\LL(v_{k+1};v_k)$ is the line segment connecting $v_{k+1}$ to $v_k$ in $\RR^{N_1N_2N_3}.$ 
Using this fact and rearranging terms in \R{m1} we have
\bbbb
\l(I-\frac{\tau_k}{2}g_v(\xi_k)\r)(v_{k+1}-v_k) &=& \tau_k\l(Cv_k+\frac{1}{2}g(v_k)\r)+\OO\l(\tau_k^2\r),
\eeee
and thus,
\bbbb
v_{k+1}-v_k & = & \tau_k\l(I-\frac{\tau_k}{2}g_v(\xi_k)\r)^{-1}\l(Cv_k+\frac{1}{2}g(v_k)\r)+\OO\l(\tau_k^2\r).
\eeee

We now proceed by induction. Letting $k=0,$ we have
\bbbb
v_{1}-v_0 & = & \tau_0\l(I-\frac{\tau_0}{2}g_v(\xi_0)\r)^{-1}\l(Cv_0+\frac{1}{2}g(v_0)\r)+\OO\l(\tau_0^2\r).
\eeee
Thus, if $\tau_0$ is sufficiently small, we have $v_1-v_0 >0$ by our assumption and then Lemma 3.3. For the sake of induction, assume that the monotonicity holds 
for $k=\ell-1.$ Then we have
\bbbb
v_{\ell+1}-v_{\ell}& = & \l[\prod_{\sigma=1}^3\l(I-\frac{\tau_{\ell}}{2}M_\sigma\r)^{-1}\l(I+\frac{\tau_{\ell}}{2}
M_\sigma\r)\r]  \\
&&\times \l(v_\ell-v_{\ell-1}+\frac{\tau_{\ell}}{2}(g(v_\ell)-g(v_{\ell-1}))\r)
+\frac{\tau_{\ell}}{2}(g(v_{\ell+1})-g(v_\ell)).
\eeee
Note that $g(v)$ is strictly increasing since $f(v)$ is strictly increasing. Utilizing Lemmas 3.2-3.3 we find that 
$v_{\ell+1}-v_\ell>0$ if $v_\ell-v_{\ell-1}>0,$ which completes the induction.
\end{proof}

It is not uncommon to set $v_0\equiv 0$ in practical combustion simulations. The following corollary 
shows that in this case conditions in Lemma 4.1 are satisfied for $\ell = 0.$

\begin{cor}
If $v_0\equiv 0$ and $\tau_0 < \min\l\{\beta_{\min}\min\{a^2,b^2,c^2\}, 
\min_{i,j,k}\frac{2\phi_{i,j,k}}{f_v(\xi_0(x_i,y_j,z_k))}\r\},$ then conditions {\rm(a), (b)} are true.
\end{cor}

\begin{proof}
We first consider {\rm (a)}:
$$Cv_0+\frac{1}{2}g(v_0) = \frac{1}{2}g(0)>0$$
which follows from $f(0) = f_0 > 0.$ 

We now consider {\rm (b)}, and under these circumstances we need to show
$$\l(I-\frac{\tau_0}{2}g_v(\xi_0)\r)^{-1} >0.$$
First, we note that $g_v(\xi_0)$ is diagonal by definition, since
$$g(v) = \l(g_{1,1,1},\dots,g_{N_1,N_2,N_3}\r)^{\tT} =  
\l(\frac{f(v_{1,1,1})}{\phi_{1,1,1}},\dots,\frac{f(v_{N_1,N_2,N_3})}{\phi_{N_1,N_2,N_3}}\r)^{\tT}$$
and
$$g_v(v) =  \l(\begin{array}{ccc} \frac{\partial g_{1,1,1}}{\partial v_{1,1,1}} & \cdots & 
\frac{\partial g_{1,1,1}}{\partial v_{N_1,N_2,N_3}}\\ \vdots & \ddots & \vdots\\ 
\frac{\partial g_{N_1,N_2,N_3}}{\partial v_{1,1,1}} & \cdots & \frac{\partial 
g_{N_1,N_2,N_3}}{\partial v_{N_1,N_2,N_3}} \end{array}\r) = 
\mbox{diag}\l(\frac{f_v(v_{1,1,1})}{\phi_{1,1,1}},\dots,\frac{f_v(v_{N_1,N_2,N_3})}{\phi_{N_1,N_2,N_3}}\r).$$
Let us denote
$$g_v(\xi_0)
= \mbox{diag}\l(\frac{f_v((\xi_0)_{1,1,1})}{\phi_{1,1,1}},\dots,
\frac{f_v((\xi_0)_{N_1,N_2,N_3})}{\phi_{N_1,N_2,N_3}}\r)
 = \mbox{diag}\l(d_{1,1,1}^{(0)},\dots,d_{N_1,N_2,N_3}^{(0)}\r).$$
It follows readily that
$$\l(I-\frac{\tau_0}{2}g_v(\xi_0)\r)^{-1} = \mbox{diag}\l(\frac{2}{2-{\tau_0}d_{1,1,1}^{(0)}},\dots,
\frac{2}{2-{\tau_0}d_{N_1,N_2,N_3}^{(0)}}\r),$$
and (b) holds if
$$\tau_0 d_{i,j,k}^{(0)}< 2,~~~1\le i\le N_1,~1\le j\le N_2,~1\le k\le N_3.$$
Denote $d^{(0)}\equiv \max_{i,j,k}\l\{ d_{i,j,k}^{(0)}\r\},$ then $\tau_0 d^{(0)}< 2$ which leads to (b).
\end{proof}

\begin{lemma}
For any $\tau_\ell > 0$ we have
$$\l(I-\frac{\tau_\ell}{2}M_\sigma\r)x \ge x,~~~\sigma = 1,2,3,$$
where $x=(1,1,\dots,1)^{\tT}.$
\end{lemma}

\begin{proof}
We only need to show the case with
$$w = \l(I-\frac{\tau_\ell}{2}M_1\r)x = (w_{1,1,1},\dots,w_{i,j,k},\dots,w_{N_1,N_2,N_3})^{\tT}.$$
First, we observe that
\bbbb
w_{1,1,1} & = & \l(1-\frac{\tau_\ell}{2}\cdot\frac{-2}{a^2\phi_{1,1,1} h_{1,0}h_{1,1}}\r) - \frac{\tau_\ell}{2}\cdot\frac{2} {a^2\phi_{1,1,1}h_{1,1}(h_{1,0}+h_{1,1})}\\
& = & 1 + \frac{\tau_\ell}{a^2\phi_{1,1,1}}\l(\frac{1}{h_{1,0}h_{1,1}} - \frac{1}{h_{1,1}(h_{1,0}+h_{1,1})}\r)
 ~>~  1.
\eeee
Second, for $i=2,\dots,N_1-1$ we have
\bbbb
w_{i,1,1} & = & -\frac{\tau_\ell}{2}\cdot\frac{2}{a^2\phi_{i,1,1}h_{1,i-1}(h_{1,i-1}+h_{1,i})} + 
\l(1-\frac{\tau_\ell}{2}\cdot\frac{-2}{a^2\phi_{i,1,1} h_{1,i-1}h_{1,i}}\r)\\
&& - \frac{\tau_\ell}{2}\cdot\frac{2} {a^2\phi_{i,1,1}h_{1,i}(h_{1,i-1}+h_{1,i})}\\
& = & 1 + \frac{\tau_\ell}{a^2\phi_{i,1,1}}\l[\frac{-h_{1,i} + (h_{1,i-1}+h_{1,i}) - h_{1,i-1}}{h_{1,i-1}h_{1,i}(h_{1,i-1}+h_{1,i})}\r]
~ = ~ 1.
\eeee
Third, we have
\bbbb
w_{N_1,1,1} & = & -\frac{\tau_\ell}{2}\cdot\frac{2}{a^2\phi_{N_1,1,1}h_{1,N_1-1}(h_{1,N_1-1}+h_{1,N_1})} + 
\l(1-\frac{\tau_\ell}{2}\cdot\frac{-2}{a^2\phi_{N_1,1,1} h_{1,N_1-1}h_{1,N_1}}\r)\\
& = & 1 + \frac{\tau_\ell}{a^2\phi_{N_1,1,1}}\l[\frac{1}{h_{1,N_1}(h_{1,N_1-1}+h_{1,N_1})}\r] ~ > ~ 1.
\eeee
Hence, we conclude that $w_{i,1,1} \ge 1,~i=1,\dots,N_1.$ Similar arguments may show that all remaining 
elements of $w$ are also bounded below by 1. Therefore we have $w\ge x.$ Similar discussions may be 
utilized for the cases involving $M_2$ or $M_3.$
\end{proof}

\vspace{2mm}

In the next lemma we show that numerical quenching, {\em i.e.,\/} one or more components of $v_\ell$ reaching or 
exceeding unity, cannot occur immediately after the first time step under appropriate constraints. To this end, we denote 
$h=\max_{j=1,\ldots,N_{\sigma},~\sigma=1,2,3}\l\{h_{\sigma,j}\r\}.$

\begin{lemma}
If \R{cfl} holds and $h^2 < \frac{1}{2\min\{a^2,b^2,c^2\}}\min\l\{\frac{1}{f_0},
\frac{4}{f(\tau_0f_0/\phi_{\min})}\r\},$ then for given $v_0\equiv 0$, we have that all components of $v_1<1.$
\end{lemma}

\begin{proof}
If $v_0\equiv 0,$ then from \R{c3} we have
$$v_1 = \l[\prod_{\sigma=1}^3\l(I-\frac{\tau_0}{2}M_\sigma\r)^{-1}\l(I+\frac{\tau_0}{2}M_\sigma\r)\r]\frac{\tau_0}{2}g(0) + 
\frac{\tau_0}{2}g(v_1).$$
Using
$$g(v_1) \approx g(w_0) = g(v_0 + \tau_0(Cv_0+g(v_0))) = g(\tau_0f_0\xi),$$
where $\xi = \l(\phi_{1,1,1}^{-1},\dots,\phi_{N_1,N_2,N_3}^{-1}\r)^{\tT}\in\RR^{N_1N_2N_3},$ we have following 
decomposed connections
\bbb
\l(I-\frac{\tau_0}{2}M_1\r)\tilde{v}_0 & = & \l(I+\frac{\tau_0}{2}M_1\r)\frac{\tau_0}{2}f_0\xi,\label{37a}\\
\l(I-\frac{\tau_0}{2}M_2\r)\bar{v}_0 & = & \l(I+\frac{\tau_0}{2}M_2\r)\tilde{v}_0,\label{37b}\\
\l(I-\frac{\tau_0}{2}M_3\r)\l(v_1-\frac{\tau_0}{2}g(\tau_0f_0\xi)\r) & = & \l(I+\frac{\tau_0}{2}M_3\r)\bar{v}_0.\label{37c}
\eee
From \R{37a} we observe that
\bbbb
\tilde{v}_0 - \frac{1}{4}x & = & \l(I-\frac{\tau_0}{2}M_1\r)^{-1}\l[\l(I+\frac{\tau_0}{2}M_1\r)
\frac{\tau_0}{2}f_0\xi - \frac{1}{4}\l(I-\frac{\tau_{0}}{2}M_1\r)x\r]\\
& = & \l(I-\frac{\tau_0}{2}M_1\r)^{-1}\l(s_1^+ + s_1^-\r)
\eeee
for which
\bbbb
|s_1^+| & = & \l|\l(I+\frac{\tau_0}{2}M_1\r)\frac{\tau_0f_0}{2}\xi\r| ~\leq ~ \frac{\tau_0f_0}{2}\phi_{\min}^{-1}\l\|I+
\frac{\tau_0}{2}M_1\r\|_2\\
& < & \tau_0f_0\phi_{\min}^{-1} ~ < ~ \frac{h^2f_0}{2\|B\|_2}\min\l\{a^2,b^2,c^2\r\}\phi_{\min}^{-1} ~ \le ~ 
\frac{h^2f_0}{2}\min\l\{a^2,b^2,c^2\r\}.
\eeee
The above indicates that
$$s_1^+ \le \frac{h^2f_0}{2}\min\l\{a^2,b^2,c^2\r\}x.$$
On the other hand, according to Lemma 4.3 we have
$$s_1^- \le -\frac{1}{4}x,$$
and thus,
$$s_1^+ + s_1^- \le \l(\frac{h^2f_0}{2}\min\l\{a^2,b^2,c^2\r\} - \frac{1}{4}\r)x.$$
Since we wish each component of $s_1^+ + s_1^-$ to be negative, we require
\bb{v1a}
\frac{h^2f_0}{2}\min\l\{a^2,b^2,c^2\r\}-\frac{1}{4} <  0,~~\mbox{or}~~h  <  \frac{1}{\sqrt{2f_0\min\l\{a^2,b^2,c^2\r\}}}.
\ee
Now, recall \R{37b}. It follows that
\bbbb
\bar{v}_0 - \frac{1}{2}x & = & \l(I-\frac{\tau_0}{2}M_2\r)^{-1}\l[\l(I+\frac{\tau_0}{2}M_2\r)\tilde{v}_0 - \frac{1}{2}x\r]\\
& = & \l(I-\frac{\tau_0}{2}M_2\r)^{-1}\l(s_2^+ + s_2^-\r).
\eeee
Note that
$$\l|s_2^+\r|  =  \l|\l(I+\frac{\tau_0}{2}M_2\r)\tilde{v}_0\r| <  \frac{1}{4}\l\|I+\frac{\tau_0}{2}M_2\r\|_2\\ \leq \frac{1}{2},$$
which implies that $s_2^+ < \frac{1}{2}x.$ Therefore we arrive at
$$s_2^+ + s_2^-  <  \frac{1}{2}x - \frac{1}{2}x  = 0.$$
By the same token, based on \R{37c} we observe that
\bbbb
v_1 - x & = & \l(I-\frac{\tau_0}{2}M_3\r)^{-1}\l[\l(I+\frac{\tau_0}{2}M_3\r)\bar{v}_0 + \l(I-\frac{\tau_0}{2}M_3\r)\l(\frac{\tau_0}{2}g(\tau_0f_0\xi)-x\r)\r]\\
& = & \l(I-\frac{\tau_0}{2}M_3\r)^{-1}\l(s_3^++s_3^-\r).
\eeee
It can be seen that
\bbbb
\l|s_3^+\r| & = & \l|\l(I+\frac{\tau_0}{2}M_3\r)\bar{v}_0+\l(I-\frac{\tau_0}{2}M_3\r)\frac{\tau_0}{2}g(\tau_0f_0\xi)\r|\\
& \le & \max\l\{|\bar{v}_0|,\l|\frac{\tau_0}{2}g(\tau_0f_0\xi)\r|\r\}\l\|\l(I+\frac{\tau_0}{2}M_3\r)+\l(I-\frac{\tau_0}{2}M_3\r)\r\|_2\\
& < & \max\l\{1,\frac{h^2}{2}\|B\|_2^{-1}\min\l\{a^2,b^2,c^2\r\}f(\tau_0f_0\phi_{\min}^{-1})\phi_{\min}^{-1}\r\}\\
& = & \max\l\{1,\frac{h^2f(\tau_0f_0\phi_{\min}^{-1})}{2}\min\l\{a^2,b^2,c^2\r\}\r\},
\eeee
and the above indicates that
$$s_3^+ \le \max\l\{1,\frac{h^2f(\tau_0f_0\phi_{\min}^{-1})}{2}\min\l\{a^2,b^2,c^2\r\}\r\}x.$$
By Lemma 4.2 we conclude that $s_3^- \le -x,$ and therefore,
\bbbb
s_3^+ + s_3^- & \le & \max\l\{1,\frac{h^2f(\tau_0f_0\phi_{\min}^{-1})}{2}\min\l\{a^2,b^2,c^2\r\}\r\}x - x\\
& = & \max\l\{0,\frac{h^2f(\tau_0f_0\phi_{\min}^{-1})}{2}\min\l\{a^2,b^2,c^2\r\}-1\r\}x.
\eeee
Since we again wish each component of the above vector to be negative, we need
$$\frac{h^2f(\tau_0f_0\phi_{\min}^{-1})}{2}\min\l\{a^2,b^2,c^2\r\}-1 < 0,~~\mbox{or}~~ h^2  
< \frac{2}{f(\tau_0f_0/\phi_{\min})\min\l\{a^2,b^2,c^2\r\}}.$$
Hence $v_1-x\le 0$ follows immediately from \R{v1a} and the above. \end{proof}

\begin{rem}
We could generalize the above lemma to include nonzero initial vectors, if desired. Let $0<v_0<1/8$ be given. If \R{cfl} holds and 
$h^2 < \l(1-8\max\{v_0\}\r)/\l(2F\min\{a^2,b^2,c^2\}\r),$ where $F = f\l(\|v_0\|_2+
\tau_0(C\|v_0\|_2\r.$\\
$\l.+f(\|v_0\|_2\phi_{\min}^{-1}))\r),$ then all components of $v_1$ 
generated by \R{c3} are bounded above by unity. This follows by modifying the proof of Lemma 4.4.
\end{rem}

Combining above results we obtain the following theorem.

\begin{thm}
For any beginning step $\ell_0\geq 0$ if $\tau_\ell$ 
is sufficiently small for $\ell\ge \ell_0$ and
\begin{description}
\item[{\rm(i)}] \R{cfl} holds for all $\ell\ge \ell_0,$
\item[{\rm(ii)}] $h^2 < \frac{1}{2\min\{a^2,b^2,c^2\}}\min\l\{\frac{1}{f_0},
\frac{4}{f(\tau_0f_0/\phi_{\min})}\r\},$ where $h=\max_{j=1,\dots,N_\sigma,~\sigma=1,2,3}\{h_{\sigma,j}\},$
\item[{\rm(iii)}] $Cv_{\ell_0}+\frac{1}{2}g(v_{\ell_0})>0$ and 
$\l(I-\frac{\tau_{\ell_0}}{2}g_v(\xi_{\ell_0})\r)^{-1}>0,$
\end{description}
then the sequence $\l\{v_\ell\r\}_{\ell\ge {\ell_0}}$ produced by the semi-adaptive 
LOD scheme \R{c3} increases monotonically until unity is reached or exceeded by 
one or more components of the solution vector, i.e., until quenching occurs.
\end{thm}

\vspace{2mm}

\section{Stability} \clearallnum
Nonlinear stability has been an extremely difficult issue when nonlinear Kawarada equations are concerned
\cite{Acker2,Beau1,Cao2,Sheng21,Sheng15b,Sheng3}. However, when the numerical solution varies relatively slowly, 
that is, before reaching a certain neighborhood of quenching, instability may be detected through a linear stability 
analysis of the nonlinear scheme utilized \cite{Cheng,Lang2,Twi}. Although the application of such an analysis 
to nonlinear problems cannot be rigorously justified, it has been found to be remarkably informative in practical 
computations. In the following study, we will first carry out a linearized stability analysis in the von Neumann sense
for \R{c3} with its 
nonlinear source term frozen. This is equivalent to assuming that the source term is effectively accurate. 
The analysis will then be extended to circumstances where the nonlinear term is not frozen. In the later case, 
the boundedness of the Jacobian of the source term, $\|g_v(v)\|_2,$ which is equivalent to assuming that we are some neighborhood away from quenching, is assumed.

In the following, let $A\in\mathbb{C}^{n\times n}$ and again denote $E(\cdot) = \exp(\cdot)$ for $n>1.$ 

\begin{define}
Let $\|\cdot\|$ be an induced matrix norm. Then the 
associated logarithmic norm $\mu : \mathbb{C}^{n\times n}\to \RR$ of $A$ is defined as
$$\mu(A) = \lim_{h\to 0^+} \frac{\|I_n + hA\| - 1}{h},$$
where $I_n\in\mathbb{C}^{n\times n}$ is the identity matrix.
\end{define}

\begin{rem}
If the matrix norm being considered is the spectral norm, 
then $\mu(A) = \max\left\{\lambda : \lambda\ \text{is an eigenvalue of}\ (A+A^*)/2\right\} = \frac{1}{2}\lambda_{\max}(A+A^*).$
\end{rem}

\begin{lemma}
For $\alpha\in\mathbb{C}$ we have 
$$\|E(\alpha A)\| \le E(\alpha \mu(A)).$$
\end{lemma}

\begin{proof}
See \cite{Golub}.
\end{proof}

For the semi-adaptive LOD method \R{c3} with its nonlinear source term frozen, regularity 
conditions need to be imposed upon the nonuniform spatial grids for a linear
stability analysis. For this purpose, let us denote $h_{\sigma} = \min_{j=1,\dots,N_\sigma}\{h_{\sigma,j}\},~
\sigma=1,2,3.$

\begin{lemma}
If
\bbb
\frac{1}{h_{1}^2\phi_{i-1,j,k}} - \frac{1}{h_{1,i-1}h_{1,i}\phi_{i,j,k}} &\leq &\frac{K}{2},\label{sss1}\\
\frac{1}{h_{2}^2\phi_{i,j-1,k}} - \frac{1}{h_{2,j-1}h_{2,j}\phi_{i,j,k}} &\leq &\frac{K}{2},\label{sss2}\\
\frac{1}{h_{3}^2\phi_{i,j,k-1}} - \frac{1}{h_{3,k-1}h_{3,k}\phi_{i,j,k}} &\leq &\frac{K}{2},\label{sss3}
\eee
where the constant $K>0$ is independent of $h_{\sigma,j},~j=1,\dots,N_{\sigma},~\sigma=1,2,3.$ 
then $$\mu(M_\sigma)\le K,~\sigma=1,2,3.$$
\end{lemma}

\begin{proof}
We only need to consider the case involving $M_1$ since the
other cases are similar. Note that $\mu(M_1)= \frac{1}{2}\lambda_{\max}\l(M_1+M_1^{\tT}\r)$ and
$$\frac{1}{2}\l(M_1+M_1^{\tT}\r) = \mbox{diag}(X_{1,1},\dots,X_{N_2,1},X_{1,2},\dots,X_{N_2,N_3})\in\RR^{N_1N_2N_3\times N_1N_2N_3},$$
where
$$\l(X_{j,k}\r)_{n,p} = \l\{\begin{array}{cl} \frac{m_{1,n}}{\phi_{n,j,k}}, &\mbox{if}\ n=p,\\ 
\frac{n_{1,n-1}}{2\phi_{n-1,j,k}}+\frac{l_{1,n-1}}{2\phi_{n,j,k}}, &\mbox{if}\ n-p=1,\\ 
\frac{n_{1,n}}{2\phi_{n,j,k}}+\frac{l_{1,n}}{2\phi_{n+1,j,k}}, &\mbox{if}\ p-n=1,\\ 0, &\mbox{otherwise.} \end{array}\r. $$
We apply Ger\u{s}chgorin's circle theorem to an arbitrary $X_{j,k}$ and note that a similar argument works for each $X_{j,k},~j=1,\dots,N_2,~k=1,\dots,N_3.$ Further, notice that we only need to consider circumstances 
where the bandwidth of $M_1+M_1^{\tT}$ is three. Thus,
\bbbb
\l|\lambda_{1,i} - \frac{m_{1,i}}{\phi_{i,j,k}}\r| & \le & \l|\frac{n_{1,i-1}}{2\phi_{i-1,j,k}}+\frac{l_{1,i-1}}{2\phi_{i,j,k}}\r|+
\l|\frac{n_{1,i}}{2\phi_{i,j,k}}+\frac{l_{1,i}}{2\phi_{i+1,j,k}}\r|
~\le~  \frac{2}{h_1^2\phi_{i-1,j,k}},\\&&i=2,\dots,N_1-1,~j=1,\dots,N_2,~k=1,\dots,N_3.
\eeee
We then see that \R{sss1} follows immediately from the above and the fact that
$$\frac{2}{h_1^2\phi_{i-1,j,k}} - \frac{2}{h_{1,i-1}h_{1,i}\phi_{i,j,k}} \leq K,~~~i=2,\dots,N_1-1,~j=1,\dots,N_2,~k=1,\dots,N_3.$$
\end{proof}

\begin{lemma}
If \R{sss1}-\R{sss3} hold then
\bb{bd1}
\l\|\l(I-\frac{\tau_\ell}{2}M_{\sigma}\r)^{-1}\l(I+\frac{\tau_\ell}{2}M_{\sigma}\r)\r\|_2 \le 1 + 
\tau_\ell K + \OO\l(\tau_\ell^2\r),~~~\ell\geq 0,~\sigma=1,2,3,
\ee
for sufficiently small $\tau_\ell>0.$
\end{lemma}

\begin{proof}
Recalling the [1/1] Pad{\'e} approximation utilized
in Section 2, we have
$$\l(I-\frac{\tau_\ell}{2}M_\sigma\r)^{-1}\l(I+\frac{\tau_\ell}{2}M_\sigma\r) = E(\tau_\ell M_\sigma) + \OO\l(\tau_\ell^3\r),~~~\sigma=1,2,3.$$
Now, based on Lemmas 5.1 and 5.2,
\bbbb
\l\|\l(I-\frac{\tau_\ell}{2}M_\sigma\r)^{-1}\l(I+\frac{\tau_\ell}{2}M_\sigma\r)\r\|_2 
& \le & E(\tau_\ell \mu(M_\sigma)) + \OO\l(\tau_\ell^3\r)\\
& \le &   
\l[1 + \tau_\ell K + \OO\l(\tau_\ell^2\r)\r] + \OO\l(\tau_\ell^3\r)\\
& = & 1 + \tau_\ell K + \OO\l(\tau_\ell^2\r),
\eeee
which is the desired bound.
\end{proof}

Combining the above results gives the following theorem.

\begin{thm}
If \R{cfl} and \R{sss1}-\R{sss3} hold, then the semi-adaptive LOD 
method \R{c3} with the source term frozen is unconditionally stable in the von Neumann sense 
under the spectral norm, that is,
$$\|z_{\ell+1}\|_2 \leq c \|z_{0}\|_2,~~~\ell\geq 0, $$
where $z_0=v_0-\tilde{v}_0$ is an initial error, $z_{\ell+1}=v_{\ell+1}-\tilde{v}_{\ell+1}$ is the 
$(\ell+1)$th perturbed error vector, and $c>0$ is a constant independent of $\ell$ and $\tau_\ell.$
\end{thm}

\begin{proof}
When the nonlinear source term is frozen, $z_{\ell+1}$ takes the form of
\bb{st1}
z_{\ell+1} = \prod_{\sigma=1}^3\l(I-\frac{\tau_\ell}{2}M_\sigma\r)^{-1}\l(I+\frac{\tau_\ell}{2}M_\sigma\r) z_\ell,~~~\ell\geq 0.
\ee
Recall that $\sum_{k=0}^{\ell}\tau_k\leq T,~\ell>0.$ It follows by taking the norm on both sides of \R{st1} that
\bbbb
\|z_{\ell+1}\|_2 & \le & \prod_{\sigma=1}^3\l\|\l(I-\frac{\tau_\ell}{2}M_\sigma\r)^{-1}\l(I+\frac{\tau_\ell}{2}
M_\sigma\r)\r\|_2\|z_\ell\|_2\\ 
& \leq & \l(1 + 3\tau_\ell K + c_2\tau_\ell^2\r)\|z_\ell\|_2 ~ \leq ~ \prod_{k=0}^\ell \l(1 + 3\tau_k K + c_3\tau_k^2\r)\|z_0\|_2\\
& \leq & \l(1 + 3KT + c_4\sum_{k=0}^\ell \tau_k^2\r)\|z_0\|_2 ~ \leq ~ c\|z_0\|_2,
\eeee
where $c_1,c_2,c_3,c_4$ and $c$ are positive constants independent of $\ell,~\tau_{k},~0\leq k\leq\ell.$ 
Therefore the theorem is clear.
\end{proof}

We now consider the case without freezing the nonlinear source term in \R{c3}. In this situation, restrictions upon the 
Jacobian matrix $g_v(v)$ become necessary.

\begin{thm}
Let $\tau_k,~0\le k\le \ell,$ be sufficiently small and \R{cfl}, \R{sss1}-\R{sss3} hold. 
If there exists a constant $G<\infty$ such that
\bb{cond1}
\|g_v(\xi)\|_2 \le G,~~\xi\in\RR^{N_1N_2N_3},
\ee
then the semi-adaptive LOD method \R{c3} is unconditionally stable in the von Neumann sense, that is,
$$ \|z_{\ell+1}\|_2 \leq \tilde{c}\, \|z_{0}\|_2,~~~\ell>0, $$
where $z_0=v_0-\tilde{v}_0$ is an initial error, $z_{\ell+1}=v_{\ell+1}-\tilde{v}_{\ell+1}$ is the 
$(\ell+1)$th perturbed error vector, and $\tilde{c}>0$ is a constant independent of $\ell$ and $\tau_\ell.$
\end{thm}

\begin{proof}
By definition we have
\bbbb
v_{\ell+1} & = & \prod_{\sigma=1}^3\l(I-\frac{\tau_\ell}{2}M_\sigma\r)^{-1}\l(I+\frac{\tau_\ell}{2}M_\sigma\r)
\l(v_\ell+\frac{\tau_\ell}{2}g(v_\ell)\r) + \frac{\tau_\ell}{2}g(v_{\ell+1})\\
& = & \Phi_\ell \l(v_\ell+\frac{\tau_\ell}{2}g(v_\ell)\r) + \frac{\tau_\ell}{2}g(v_{\ell+1}),
\eeee
where
$$\Phi_\ell = \prod_{\sigma=1}^3\l(I-\frac{\tau_\ell}{2}M_\sigma\r)^{-1}\l(I+\frac{\tau_\ell}{2}M_\sigma\r).$$
It follows that
\bbbb
z_{\ell+1} & = & \Phi_\ell z_\ell + \frac{\tau_\ell}{2}\Phi_\ell(g(v_\ell) - g(\tilde{v}_\ell)) + \frac{\tau_\ell}{2}(g(v_{\ell+1})-g(\tilde{v}_{\ell+1}))\\
& = & \Phi_\ell z_\ell + \frac{\tau_\ell}{2}\Phi_\ell g_v(\xi_\ell)z_\ell + \frac{\tau_\ell}{2}g_v(\xi_{\ell+1})z_{\ell+1},\\
\eeee
where $\xi_k\in\LL(v_k,\tilde{v}_k),~k=\ell,\ell+1.$ Rearranging the above equality, we have
$$\l(I-\frac{\tau_\ell}{2}g_v(\xi_{\ell+1})\r)z_{\ell+1}  
= \Phi_\ell\l(I+\frac{\tau_\ell}{2}g_v(\xi_\ell)\r)z_\ell.$$
Further, recall \R{cond1}. When $\tau_k$ is sufficiently small we may claim that
$$\l(I-\frac{\tau_k}{2}g_v(\xi)\r)^{-1},~I+\frac{\tau_k}{2}g_v(\xi)
=E\l(\frac{\tau_k}{2}g_v(\xi)\r)+\OO\l(\tau_k^2\r).$$
Thus,
\bbbb
z_{\ell+1} & = & \l(I-\frac{\tau_\ell}{2}g_v(\xi_{\ell+1})\r)^{-1}\Phi_\ell\l(I+\frac{\tau_\ell}{2}g_v(\xi_\ell)\r)z_\ell\nnn\\
& =& \l\{\prod_{k=0}^\ell \l[E\l(\frac{\tau_k}{2}g_v(\xi_{k+1})\r)+\OO\l(\tau_k^2\r)\r]\Phi_k \l[E\l(\frac{\tau_k}{2}g_v(\xi_k)\r)+\OO\l(\tau_k^2\r)\r]\r\}z_0.\label{rec1}
\eeee
It follows therefore
\bbbb
\|z_{\ell+1}\|_2 & \le & \l\|\Phi_k\r\|_2\l\{\prod_{k=0}^\ell \l\|E\l(\frac{\tau_k}{2}g_v(\xi_{k+1})\r)\r\|_2\l\|E\l(\frac{\tau_k}{2}g_v(\xi_k)\r)\r\|_2+c_{1,k}\tau_k^2 \r\}\|z_0\|_2\nnn\\
& \le & \l(1 + 3KT + c\sum_{k=0}^{\ell} \tau_k^2\r)\l(e^{GT} + c_1\sum_{k=0}^\ell \tau_k^2\r)\|z_0\|_2 ~ \le ~ \tilde{c}\,\|z_0\|_2,\label{stb1}
\eeee
where $c_{1,k},~k=1,2,\ldots,\ell,$ are positive constants and 
$c,c_1,\tilde{c}$ are positive constants independent of $\ell$ and $\tau_{\ell},~\ell>0.$ 
Thus giving the desired stability.
\end{proof}

The above theorem provides further insight as to why the standard linear analysis can be useful in estimating the nonlinear stability. 
The extra cost paid, however, is assuming the boundedness of $\|g_v(\xi)\|_2.$ Nevertheless, this is an improvement upon the
traditional methodology of having the nonlinear source term frozen. In fact, the aforementioned bound is well-maintained 
in numerical experiments until certain neighborhoods of quenching are reached. This serves as an indication that the new
analysis is valid and effective.

\section{Conclusions} \clearallnum
A semi-adaptive LOD scheme is developed for solving degenerate Kawarada equations possessing a strong quenching nonlinearity 
and singularity. While a temporal adaptation is performed via an arc-length monitoring mechanism of the temporal derivative of the solution, fixed nonuniform spatial grids are adopted. The novel splitting method is implicit and 
the impact of the degeneracy is found to be limited. Rigorous analysis is given for key computational
features, including the positivity, monotonicity, and stability, of the numerical solution.  Important criteria to guarantee these properties, which depend upon the variable steps and degenerate function, are obtained. 

Under much weaker requirements (see the latest results in \cite{Beau1}), 
the temporal step restriction for guaranteeing monotone numerical solutions of our LOD scheme has been reduced to 
only one-half of those in uniform spatial mesh cases \cite{Sheng21}. 
Furthermore, a realistic method of targeting the realization of nonlinear stability analysis is proposed and shown to be successful. 
Though this new strategy needs the 
boundedness of $\|g_v(\xi)\|_2,$ the requirement is well-justified before quenching is reached. This improved methodology not only provides further insight into the stability,
but also offers explanations as to why the linear stability analysis must be valid before quenching. 
On the other hand, simulations of real three-dimensional
solutions still remain as one of the most challenging tasks. 
In anticipated future work we plan to utilize the latest
{\em High Performance Computing\/}
tools with large data computations for this purpose. More rigorous and generalized analysis, as well as
non-exponential splitting based higher order splitting methods 
\cite{Sheng1,Sheng15b}  
will also be be investigated, studied, and experimented with.

\vspace{8mm}

\no
\end{document}